\documentclass[a4paper,11pt]{article}

\usepackage[lmargin=30mm,rmargin=30mm]{geometry}
\usepackage{amsmath,amssymb}
\usepackage{subcaption}
\usepackage{booktabs}
\usepackage{pdflscape}
\usepackage[ruled]{algorithm2e}

\usepackage{tikz}
\tikzstyle{vertex}=[circle,fill=black!25,minimum size=10pt,inner sep=0pt]
\tikzstyle{edge} = [draw,thick,-]

\graphicspath{{./figures/}}

\def\NP{{\cal NP}}

\def\RR{\mathbb R}
\newcommand{\degree}[2][]{\mathrm{deg}_{#1}\left(#2\right)}

\newenvironment{proofsketch}
 {\begin{trivlist} \item[] {\bf Proof (sketch).\ }}{\hfill$\Box$ \end{trivlist}}

\newtheorem{conjecture}{Conjecture}[section]
\newtheorem{theorem}{Theorem}[section]
\newtheorem{lemma}{Lemma}[section]
\newtheorem{proposition}{Proposition}[section]
\newenvironment{proof}
 {\begin{trivlist} \item[] {\bf Proof.\ }}{\hfill$\Box$ \end{trivlist}}

\title{\vspace{-20mm}A Two-Level Graph Partitioning Problem Arising\\
in Mobile Wireless Communications}

\author{%
  Jamie Fairbrother\thanks{STOR-i Centre for Doctoral Training,
Lancaster University, Lancaster LA1 4YF, UK. E-mail: {\tt j.fairbrother@lancaster.ac.uk}} %
  \and Adam N.\ Letchford\thanks{Department of Management Science, Lancaster University, Lancaster LA1 4YW, United Kingdom. E-mail: {\tt a.n.letchford@lancaster.ac.uk}}
  \and Keith Briggs\thanks{Wireless Research Group, BT Technology, Service \& Operations, Martlesham Heath, UK.  Email: {\tt keith.briggs@bt.com}}
}

\date{May 2017}

\begin{document}

\maketitle

\begin{abstract}
\noindent
In the {\em $k$-partition problem} ($k$-PP), one is given an edge-weighted undirected graph, and one
must partition the node set into at most $k$ subsets, in order to minimise (or maximise) the total weight
of the edges that have their end-nodes in the same cluster. Various hierarchical variants of this problem
have been studied in the context of data mining. We consider a `two-level' variant that arises in mobile
wireless communications. We show that an exact algorithm based on intelligent preprocessing, cutting planes
and symmetry-breaking is capable of solving small- and medium-size instances to proven optimality, and
providing strong lower bounds for larger instances.
\\*[2mm]
{\bf Keywords}: graph partitioning, integer programming, cutting planes, telecommunications.
\end{abstract}

\section{Introduction}

Telecommunications has proven to be a rich source of interesting optimisation problems \cite{resende2007handbook}.
In the case of {\em wireless}\/ communications, the hardest (and most strategic) problem is
{\em wireless network design}, which involves the simultaneous
determination of cell locations and shapes, base station locations, power levels and frequency
channels (e.g., \cite{mannino2007}). On a more tactical level, one finds various
{\em frequency assignment}\/ problems, which are concerned solely with the assignment of available
frequency bands to wireless devices (e.g., \cite{aardal2007models}).

Recently, we came across a rather different optimisation problem, in the
context of {\em mobile}\/ wireless communications.  The technology is the new
4G (LTE) standard, and the essence of the problem is as follows. There are a
number of devices with known locations. Each device must be assigned a positive
integer {\em identifier}.  In the LTE standard, this is called a Physical Cell
Identifier or PCI, but to keep the discussion general we will simply use the
term ID.
If two devices are close to each other (according to some measure of
closeness), they are said to be {\em neighbours}.  Two neighbouring devices
must not have the same ID. We are also given two small integers $k, k' \geqslant 2$.
If the IDs of two neighbouring devices are the same {\em modulo $k$}, it causes
interference. Moreover, some additional interference occurs if they are the same
modulo $k k'$. The task is to
assign IDs to devices in such a way as to minimise the total interference.

The problem turns out to be a generalisation of a well-known $\NP$-hard combinatorial optimisation
problem called the {\em $k$-partition problem}\/ or $k$-PP. For reasons which will become clear later,
we call our problem the {\em 2-level partition problem}\/ or 2L-PP. Here we develop an exact
algorithm for the 2L-PP, which turns out be capable of solving small - and medium-sized instances
to proven optimality, and providing strong lower bounds for larger instances.

The structure of the paper is as follows. In Section~\ref{se:lit}, the literature on the $k$-PP is reviewed.
In Section~\ref{se:theory}, we formulate our problem as an integer program (IP) and derive some
valid linear inequalities (i.e.\ cutting planes). In Section~\ref{se:algorithm}, we describe our exact
algorithm in detail. In Section~\ref{se:experiments}, we describe some computational experiments and
analyse the results. Finally, some concluding remarks are made in Section~\ref{se:conclusion}.

\section{Literature Review} \label{se:lit}

Since the 2L-PP is a generalisation of the $k$-PP, we now review the literature on the $k$-PP.
We define the $k$-PP in Subsection~\ref{sub:lit1}. The main IP formulations are presented in
Subsection~\ref{sub:lit2}. The remaining two subsections cover cutting planes and algorithms for
generating them, respectively.

We remark that some other multilevel graph partitioning problems have been studied in the data mining
literature; see, e.g., \cite{chatziafratis,sanders2011multilevel}. In those problems, however, neither the
number of clusters nor the number of levels is fixed. For this reason, we do not consider them further.

\subsection{The $k$-partition problem} \label{sub:lit1}

The $k$-PP was first defined in \cite{carlson1966scheduling}. We are given a (simple, loopless)
undirected graph $G$, with vertex set $V = \{1, \ldots, n\}$ and edge set $E$, a rational weight $w_e$
for each edge $e \in E$, and an integer $k$ with $2 \leqslant  k \leqslant  n$. The task is to partition
$V$ into $k$ or fewer subsets (called ``clusters" or ``colours''), such that the sum of the weights
of the edges that have both end-vertices in the same cluster is minimised.

The $k$-PP has applications in scheduling, statistical clustering, numerical linear algebra,
telecommunications, VLSI layout and statistical physics (see, e.g., \cite{carlson1966scheduling,eisenblatter2002semidefinite,ghaddar2011branch,rendl2012semidefinite}).
It is strongly $\NP$-hard for any fixed $k\geqslant 3$, since it includes as a special case the
problem of testing whether a graph is $k$-colourable. It is also strongly $\NP$-hard when $k=2$, since it
is then equivalent the well-known {\em max-cut}\/ problem, and when $k=n$, since it is then equivalent to the
{\em clique partitioning}\/ problem \cite{grotschel1989cutting,grotschel1990facets}.

\subsection{Formulations of the $k$-PP} \label{sub:lit2}

Chopra \& Rao \cite{chopra1993partition} present two different IP formulations for the $k$-PP.
 In the first formulation, there are two sets of binary variables. For each $v \in V$ and for
$c=1, \ldots, k$, let $x_{vc}$ be a binary variable, taking the value $1$ if and only if vertex $v$
has colour $c$. For each edge $e \in E$, let $y_e$ be an additional binary variable, taking the
value $1$ if and only if both end-nodes of $e$ have the same colour. Then we have the following optimization problem:
\begin{eqnarray}
\nonumber
\min        & \sum_{e \in E} w_e y_e & \\
\label{eq:assign}
\mbox{s.t.} &\sum_{c=1}^k x_{vc} = 1    & (v \in V) \\
\label{eq:tri1}
           & y_{uv} \geqslant x_{uc} + x_{vc} - 1  & (\{u,v\} \in E, \, c = 1, \ldots, k) \\
           & x_{uc} \geqslant x_{vc} + y_{uv} - 1  & (\{u,v\} \in E, \, c = 1, \ldots, k) \\
\label{eq:tri3}
           & x_{vc} \geqslant x_{uc} + y_{uv} - 1  & (\{u,v\} \in E, \, c = 1, \ldots, k) \\
\nonumber
            & x_{vc} \in \{0,1\}    & (v \in V, \, c = 1, \ldots, k) \\
\nonumber
            & y_{uv} \in \{0,1\}    & (\{u,v\} \in E).
\end{eqnarray}
The equations (\ref{eq:assign}) force each node to be given exactly one colour,
and the constraints (\ref{eq:tri1})--(\ref{eq:tri3}) ensure that the $y$ variables take
the value $1$ when they are supposed to.

Note that this IP has ${\cal O}(m+nk)$ variables and constraints, where $m=|E|$. It therefore
seems suitable when $k$ is small and $G$ is sparse. Unfortunately, it has a very weak linear
programming (LP) relaxation. Indeed, if we set all $x$ variables to $1/k$ and all $y$ variables
to $0$, we obtain the trivial lower bound of $0$. Moreover, it suffers from {\em symmetry}, in the
sense that given any feasible solution, there exist $k!$ solutions of the same cost.
(See Margot \cite{margot2010symmetry} for a tutorial and survey on symmetry issues in integer
programming.)

The second IP formulation is obtained by dropping the $x$ variables, but having a $y$ variable
for {\em every}\/ pair of nodes. That is, for each pair of nodes $\{u,v\}$, let $y_{uv}$ be a binary variable,
taking the value $1$ if and only if $u$ and $v$ have the same colour. Then:
\begin{eqnarray}
\nonumber
\min        & \sum_{e \in E} w_e y_e & \\
\label{eq:clq}
\mbox{s.t.} & \sum_{u, v \in C} y_{uv} \geqslant 1 & (C \subset V: |C| = k+1)\\
\label{eq:trans}
         & y_{uv} \geqslant y_{uw} + y_{vw} - 1  & (\{u,v,w\} \subset V) \\
\nonumber
           & y_{uv} \in \{0,1\}    & (\{u,v\} \subset V).
\end{eqnarray}
The constraints (\ref{eq:clq}), called {\em clique}\/ inequalities, ensure that, in any set of $k+1$
nodes, at least two receive the same colour. The constraints (\ref{eq:trans}) enforce
{\em transitivity}; that is, if nodes $u$ and $w$ have the same colour, and nodes $v$ and $w$
have the same colour, then nodes $u$ and $v$ must also have the same colour.

A drawback of the second IP formulation is that it has ${\cal O}(n^2)$ variables and ${\cal O}(n^{k+1})$
constraints, and it cannot exploit any special structure that $G$ may have (such as sparsity). 

A third IP formulation, based on so-called {\em representatives}, is studied in \cite{ales2016}.
There also exist several {\em semidefinite programming}\/ relaxations of the $k$-PP (see, e.g., \cite{frieze1997improved,eisenblatter2002semidefinite,ghaddar2011branch,rendl2012semidefinite,anjos2013solving,sotirov2013efficient,letchford2016projection}). For the sake of brevity e do not go into details.

\subsection{Cutting planes} \label{sub:lit3}

Chopra \& Rao \cite{chopra1993partition} present several families of valid linear inequalities
(i.e.\ cutting planes), which can be used to strengthen the LP relaxation of the
above formulations. For our purposes, the most important turned out to be the
\emph{generalised clique}\/ inequalities. In the case of the first IP formulation, they take the form
\begin{equation} \label{eq:gen-clq}
\sum_{u, v \in C} y_{uv} \, \geqslant \, \binom{t+1}{2} r + \binom{t}{2} (k-r),
\end{equation}
where $C \subseteq V$ is a clique (set of pairwise adjacent nodes) in $G$ with $|C| > k$, and
$t$ and $r$ denote $\big\lfloor |C|/k \big\rfloor$ and $r = |C| \bmod k$, respectively. In the
case of the second IP formulation, they must be defined for any $C \subseteq V$ with $|C| > k$
(since every set of nodes forms a clique in a complete graph). In either case, they define facets
of the associated polytope when $k \geqslant 3$ and $r \ne 0$. Note that they reduce to the
clique inequalities (\ref{eq:clq}) when $|C| = k+1$.

Further inequalities for the first IP formulation can be found in \cite{chopra1993partition,letchford2016projection}.
Further inequalities for the second formulation can be found in, e.g., \cite{grotschel1990facets,deza1990complete,deza1992clique,chopra1993partition,chopra1995facets,oosten2001clique}.

\subsection{Separation algorithms} \label{sub:lit4}

For a given family of valid inequalities, a \emph{separation algorithm} is an algorithm which
takes an LP solution and searches for violated inequalities in that family \cite{grotschel1988geometric}.

By brute-force enumeration, one can solve the separation problem for the inequalities
(\ref{eq:tri1})--(\ref{eq:tri3}) in ${\cal  O}(km)$ time, for the transitivity inequalities (\ref{eq:trans})
in ${\cal  O}\big( n^3 \big)$ time, and for the clique inequalities (\ref{eq:clq}) in ${\cal O}\big(n^{k+1}\big)$
time. It is stated in \cite{chopra1993partition} that separation of the generalised clique inequalities
(\ref{eq:gen-clq}) is $\NP$-hard. An explicit proof, using a reduction from the max-clique problem,
is given in \cite{eisenblatter2001frequency}. Heuristics for clique and generalised clique
separation are presented in \cite{eisenblatter2001frequency,kaibel2011orbitopal}.

Separation results for other inequalities for the first IP formulation can be found in
\cite{chopra1993partition,letchford2016projection}. Separation results for the second formulation
can be found in, e.g.,\cite{grotschel1989cutting,deza1992clique,caprara1996zerohalf,borndorfer2000set,oosten2001clique,letchford2001disjunctive,muller2002transitive}. For
some computational results with various separation algorithms, see \cite{desousa2016}.

\section{Formulation and Valid Inequalities} \label{se:theory}

In this section, we give an IP formulation of the 2L-PP (Subsection~\ref{sub:th-form}) and
derive some valid inequalities (\ref{sub:th-valid}). We also show how to modify the
formulation to address issues of {\em symmetry} (\ref{sub:th-sym}).

\subsection{Integer programming formulation} \label{sub:th-form}

An instance of the 2L-PP is given by an undirected graph $G=(V,E)$, integers $k, k' \geqslant 2$, and weights
$w, w' \in \RR$, with $w \ge w' > 0$. Each node in $V$ corresponds to a device, and a pair of nodes is
connected by an edge if and only if the corresponding devices are neighbours. The weights $w$ and $w'$
represent the importance given to interference modulo $k$ and modulo $k k'$, respectively.

Now, let us call the integers in $\{0, \ldots, k k' - 1\}$ {\em colours}. Assigning the colour $c$ to a node
corresponds to giving the corresponding device an ID that is congruent to $c$ modulo $k k'$. Then, the
2L-PP effectively calls for a colouring of the nodes of $G$ such that the following quantity is minimised:
$w'$ times the number of edges whose end-nodes have the same colour, plus $w$ times the number
of edges whose end-nodes have the same colour modulo $k$.

To formulate the 2L-PP as an IP, we modify the first formulation mentioned in Subsection~\ref{sub:lit2}.
We have three set of binary variables. For each $v \in V$ and for $c=0, \ldots, k k'-1$, let $x_{vc}$ be a binary
variable, taking the value $1$ if and only if vertex $v$ has colour $c$. For each edge $e \in E$, define two
binary variables $y_e$ and $z_e$, taking the value $1$ if and only if both end-nodes of $e$ have the same
colour modulo $k$, or the same colour, respectively. Then we have:
\begin{eqnarray}
\label{eq:our-obj}
\min        & w \, \sum_{\{u,v\} \in E} y_{uv} + w' \, \sum_{\{u,v\} \in E} z_{uv} & \\
\label{eq:our-x}
\mbox{s.t.} & \sum_{c=0}^{k k'-1} x_{vc} = 1    & (v \in V) \\
\label{eq:our-y}
         & y_{uv} \geqslant \sum_{r=0}^{k'-1} \left( x_{u,c+rk} + x_{v,c+rk} \right) - 1  & (\{u,v\} \in E, \, c = 0, \ldots, k-1) \\
\label{eq:our-z}
         & z_{uv} \geqslant x_{uc} + x_{vc} - 1  & (\{u,v\} \in E, \, c = 0, \ldots, k k' -1) \\
         & x_{vc} \in \{0,1\}    & (v \in V, \, c = 0, \ldots, k k' - 1) \\
         & y_{uv} \in \{0,1\}    & (\{u,v\} \in E) \\
\label{eq:our-bin-z}
         & z_{uv} \in \{0,1\}    & (\{u,v\} \in E).
\end{eqnarray}
The objective function (\ref{eq:our-obj}) is just a weighted sum of the two kinds of interference. The
constraints (\ref{eq:our-x}) state that each node must have a unique colour. The constraints
(\ref{eq:our-y}) and (\ref{eq:our-z}) ensure that the two kinds of interference occur under the stated
conditions. The remaining constraints are just binary conditions.

Note that the above IP has $kk'n + 2m$ variables and $n+k(k'+1)m$ linear constraints. In practice,
this is manageable, since $k$ and $k'$ are typically small and $G$ is typically sparse.

\subsection{Valid inequalities} \label{sub:th-valid}

Unfortunately, our IP formulation of the 2L-PP shares the same drawbacks as the first formulation
of the $k$-PP mentioned in Subsection~\ref{sub:lit2}: it has a very weak LP relaxation (giving a trivial
lower bound of zero), and it suffers from a high degree of symmetry.

To strengthen the LP relaxation, we add valid linear inequalities from three families. For a clique
$C \subset V$, let $y(C)$ and $z(C)$ denote $\sum_{\{u,v\} \subset C} y_{uv}$ and
$\sum_{\{u,v\} \subset C} z_{uv}$, respectively. The first two families of inequalities are straightforward
adaptations of the generalised clique inequalities (\ref{eq:gen-clq}) for the $k$-PP:
\begin{proposition}
The following inequalities are satisfied by all feasible solutions of the 2L-PP:
\begin{itemize}
\item ``$y$-clique" inequalities, which take the form:
\begin{equation} \label{eq:y-clq}
y(C) \, \geqslant \, \binom{t+1}{2} r + \binom{t}{2} (k-r),
\end{equation}
where $C \subseteq V$ is a clique with $|C| > k$, $t = \lfloor |C|/k \rfloor$ and $r = |C| \bmod k$;
\item ``$z$-clique" inequalities, which take the form:
\begin{equation} \label{eq:z-clq}
z(C) \geqslant \binom{T+1}{2} R + \binom{T}{2} (kk' - R),
\end{equation}
where $C \subseteq V$ is a clique with $|C| > kk'$, $T = \left\lfloor \frac{|C|}{kk'} \right\rfloor$ and $R = |C| \bmod kk'$.
\end{itemize}
\end{proposition}
\begin{proof}
This follows from the result of Chopra \& Rao \cite{chopra1993partition} mentioned in Subsection~\ref{sub:lit3},
together with the fact that, in a feasible IP solution, the $y$ and $z$ vectors are the incidence vectors of a
$k$-partition and a $k k'$-partition, respectively.
\end{proof}

The third family of inequalities, which is completely new, is described in the following theorem.
\begin{theorem} \label{th:yz-valid}
For all cliques $C \subseteq V$ with $|C| > k'$, the following ``$(y,z)$-clique" inequalities are valid:
\begin{equation} \label{eq:yz}
k' \, z(C) \; \geqslant \; y(C) - t' \binom{k'}{2} - \binom{r'}{2},
\end{equation}
where $t' = \lfloor |C| /k' \rfloor$ and $r'=|C| \bmod k'$.
\end{theorem}
\begin{proof}
See the Appendix.
\end{proof}

The following two lemmas and theorem give necessary conditions for the inequalities presented so far
to be non-dominated (i.e., not implied by other inequalities).
\begin{lemma}
A necessary condition for the $y$-clique inequality \eqref{eq:y-clq} to be non-dominated is that
$r \ne 0$.
\end{lemma}
\begin{proof}
This was already shown by Chopra \& Rao \cite{chopra1993partition}.
\end{proof}

\begin{lemma}
A necessary condition for the $(y,z)$-clique inequality \eqref{eq:yz} to be non-dominated is that
$r' \ne 0$.
\end{lemma}
\begin{proof}
Suppose that  $r'=0$. The $(y,z)$-clique inequality for $C$ can be written as:
\begin{equation} \label{eq:yz-weak}
k' \, z(C) \; \geqslant \; y(C) - |C| (k'-1)/2.
\end{equation}
Now let $v$ be an arbitrary node in $C$. The $(y,z)$-clique inequality for the set $C \setminus \{v\}$ can
be written as:
\[
k' \, z(C \setminus \{v\}) \; \geqslant \; y(C \setminus \{v\}) \, - \, (|C|-2) (k'-1)/2.
\]
Summing this up over all $v \in C$ yields
\[
k'(|C|-2) \, z(C) \; \geqslant \; (|C|-2) \, y(C) \, - \, |C| (|C|-2) (k'-1)/2.
\]
Dividing this by $|C| - 2$ yields the inequality (\ref{eq:yz-weak}).
\end{proof}

\begin{theorem} \label{th:z-redundant}
A necessary condition for the $z$-clique inequality \eqref{eq:z-clq} to be non-dominated is that
$1 < R < kk' - 1$.
\end{theorem}
\begin{proof}
See the Appendix.
\end{proof}

Our experiments with the polyhedron transformation software package {\tt PORTA} \cite{christof1997porta}
lead us to make the following conjecture:
\begin{conjecture}
The following results hold for the convex hull of 2L-PP solutions:
\begin{itemize}
\item $y$-clique inequalities \eqref{eq:gen-clq} define facets if and only if $r \ne 0$.
\item $z$-clique inequalities \eqref{eq:z-clq} define facets if and only if $1 < R < kk' - 1$.
\item $(y,z)$-clique inequalities \eqref{eq:yz} define facets if and only if $r' \ne 0$.
\end{itemize}
\end{conjecture}
In any case, we have found that all three families of inequalities work very well in practice as
cutting planes. Moreover, in our preliminary experiments, we found that the $y$-clique inequalities
were the most effective at improving the lower bound, with the $yz$-clique inequalities being the
second most effective.
\\*[3mm]
{\bf Remark:} The trivial inequality $y_e \geqslant z_e$ is also valid for all $e \in E$. In our preliminary
experiments, however, these inequalities proved to be of no value as cutting planes.

\subsection{Symmetry} \label{sub:th-sym}

Another issue to address is {\em symmetry}. Note that any permutation $\sigma$
on the set of colours $\{0,\ldots,kk'-1\}$ such that $\sigma(c) \bmod k = c \bmod k$
will preserve all $k$ and $kk'$ conflicts. Since there are $k'!$ such
permutations, for any colouring (which makes use of all available colours)
there are at least $k'!$ other colourings which yield the same cost.

One easy way to address this problem, at least partially, is given in
the following theorem:
\begin{theorem} \label{th:symmetry}
For any colour $c \in \{0, \ldots, kk'-1\}$, let $\phi(c)$ denote $\lfloor c/k \rfloor + (c \bmod k)$.
Then, one can fix to zero all variables $x_{vc}$ for which $\phi(c) \geqslant v$, while
preserving at least one optimal 2L-PP solution.
\end{theorem}
\begin{proof}
See the Appendix.
\end{proof}

\noindent
{\bf Example:} Suppose that $n \geqslant 4$ and $k = k' = 3$. Then $\phi(0), \ldots, \phi(8)$ are
$0$, $1$, $2$, $1$, $2$, $3$, $2$, $3$ and $4$, respectively. So we can fix the following variables to zero:
$x_{11}, \ldots, x_{18}$; $x_{22}$; $x_{24}, \ldots, x_{28}$; $x_{35}$, $x_{37}$, $x_{38}$ and
$x_{48}$.\hfill$\Box$

\section{Exact Algorithm} \label{se:algorithm}

We now describe an exact solution algorithm for the 2L-PP. The algorithm consists of two
main stages: {\em preprocessing}\/ and {\em cut-and-branch}. Preprocessing is described in
Subsection~\ref{sub:alg-preprocessing}, while the cut-and-branch algorithm is described in
Subsection~\ref{sub:alg-cut-and-branch}. Throughout this section, for a given set of nodes
$V' \subseteq V$, we let $G[V']$ denote the subgraph of $G$ induced by the nodes in $V'$.

\subsection{Preprocessing}
\label{sub:alg-preprocessing}

In the first stage, an attempt is made to simplify the input graph $G$
and, if possible, decompose it into smaller and simpler
subgraphs. This is via two operations, which we call {\em $k$-core
reduction}\/ and {\em block decomposition}. Although we focus on the
2L-PP the following results also apply to the $k$-PP.

A $k$-core of a graph $G$ is a maximal connected subgraph whose nodes all have degree
of at least $k$. The concept was first introduced in \cite{seidman1983network}, as a tool to
measure cohesion in social networks. An example is given in Fig.~\ref{fig:3core}, but it should
be borne in mind that, in general, a graph may have several (node-disjoint) $k$-cores. The
$k$-cores of a graph can be found easily, in ${\cal O}\big( n^2 \big)$ time, via a minor
adaptation of an algorithm given in \cite{szekeres1968inequality}. Details are given in
Algorithm~\ref{alg:k-core}.

\setlength{\unitlength}{1.15cm}
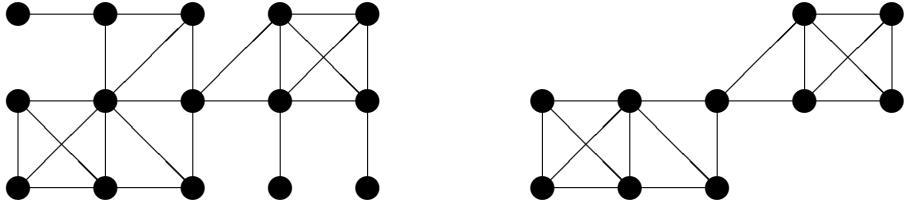
\begin{figure}
\centering
\begin{picture}(10, 3)
\put(0, 0.5){\circle*{0.27}}
\put(0, 1.5){\circle*{0.27}}
\put(0, 2.5){\circle*{0.27}}
\put(1, 0.5){\circle*{0.27}}
\put(1, 1.5){\circle*{0.27}}
\put(1, 2.5){\circle*{0.27}}
\put(2, 0.5){\circle*{0.27}}
\put(2, 1.5){\circle*{0.27}}
\put(2, 2.5){\circle*{0.27}}
\put(3, 0.5){\circle*{0.27}}
\put(3, 1.5){\circle*{0.27}}
\put(3, 2.5){\circle*{0.27}}
\put(4, 0.5){\circle*{0.27}}
\put(4, 1.5){\circle*{0.27}}
\put(4, 2.5){\circle*{0.27}}
\put(0, 0.5){\line(1,0){2}}
\put(0, 0.5){\line(0,1){1}}
\put(0, 0.5){\line(1,1){2}}
\put(0, 1.5){\line(1,0){4}}
\put(0, 1.5){\line(1,-1){1}}
\put(0, 2.5){\line(1,0){2}}
\put(1, 0.5){\line(0,1){2}}
\put(1, 1.5){\line(1,-1){1}}
\put(2, 0.5){\line(0,1){2}}
\put(2, 1.5){\line(1,1){1}}
\put(3, 1.5){\line(1,1){1}}
\put(3, 0.5){\line(0,1){2}}
\put(3, 2.5){\line(1,0){1}}
\put(3, 2.5){\line(1,-1){1}}
\put(4, 0.5){\line(0,1){2}}
\put(6, 0.5){\circle*{0.27}}
\put(6, 1.5){\circle*{0.27}}
\put(7, 0.5){\circle*{0.27}}
\put(7, 1.5){\circle*{0.27}}
\put(8, 0.5){\circle*{0.27}}
\put(8, 1.5){\circle*{0.27}}
\put(9, 1,5){\circle*{0.27}}
\put(9, 2.5){\circle*{0.27}}
\put(10, 1.5){\circle*{0.27}}
\put(10, 2.5){\circle*{0.27}}
\put(6, 0.5){\line(1,0){2}}
\put(6, 0.5){\line(0,1){1}}
\put(6, 0.5){\line(1,1){1}}
\put(6, 1.5){\line(1,0){4}}
\put(6, 1.5){\line(1,-1){1}}
\put(7, 0.5){\line(0,1){1}}
\put(7, 1.5){\line(1,-1){1}}
\put(8, 0.5){\line(0,1){1}}
\put(8, 1.5){\line(1,1){1}}
\put(9, 1.5){\line(0,1){1}}
\put(9, 1.5){\line(1,1){1}}
\put(9, 2.5){\line(1,0){1}}
\put(9, 2.5){\line(1,-1){1}}
\put(10, 1.5){\line(0,1){1}}
\end{picture}
\caption{A graph (left) and its (unique) $3$-core (right).}
\label{fig:3core}
\end{figure}

\begin{algorithm}
  \SetKwData{Left}{left}\SetKwData{This}{this}\SetKwData{Up}{up}
  \SetKwFunction{Union}{Union}\SetKwFunction{FindCompress}{FindCompress}
  \SetKwInOut{Input}{input}\SetKwInOut{Output}{output}
  \SetKwRepeat{Do}{do}{while}
  \Input{graph $G = (V,E)$}
  \Output{$k$-cores of $G$}
  $V' := V$\;
  \Do{$V^- \neq \emptyset$} {
   Let $V^-$ equal $\left\{ v\in V' :\ \degree[G\lbrack V'\rbrack ]{v} < k \right\}$\;
    \If{$V^- \ne \emptyset$}{
      $V':= V'\setminus V^-$\;
    }
  }
Output the connected components of $G[V']$\:
\caption{Algorithm for finding $k$-cores of a graph}
\label{alg:k-core}
\end{algorithm}

The reason that $k$-cores are of interest is given in the following proposition:
\begin{proposition}
For any graph $G$, the cost of the optimal 2L-PP solution is equal to the sum of the costs
of the optimal solutions of the 2L-PP instances given by its $k$-cores.
\end{proposition}
\begin{proofsketch}
Let $v \in V$ be any node whose degree in $G$ is less than $k$. Suppose we solve the 2L-PP
on the induced subgraph $G[V \setminus \{v\}]$. Then we can extend the 2L-PP solution to the
original graph $G$, without increasing its cost, by giving node $v$ a colour that is not congruent
modulo $k$ to the colour of any of its neighbours. The result follows by induction.
\end{proofsketch}

We refer to the process of the replacement of $G$ with its $k$-core(s) as
\emph{$k$-core reduction}. We will say that a graph is {\em $k$-core reducible}\/
if it is not equal to the union of its $k$-cores.

Now, a vertex of a graph is said to be an \emph{articulation point}\/ if its removal causes the graph to
become disconnected. A connected graph with no articulation points is said to be \emph{biconnected}.
The \emph{biconnected components}\/ of a graph, also called \emph{blocks}, are maximal induced
biconnected subgraphs. For example, the graph on the right of Fig.~\ref{fig:3core} has two blocks,
which are displayed on the left of Fig.~\ref{fig:blocks}. The blocks of a graph
$G=(V,E)$ can be computed in ${\cal O}(|V|+|E|)$ time \cite{hopcroft1973algorithm}.

\setlength{\unitlength}{1.15cm}
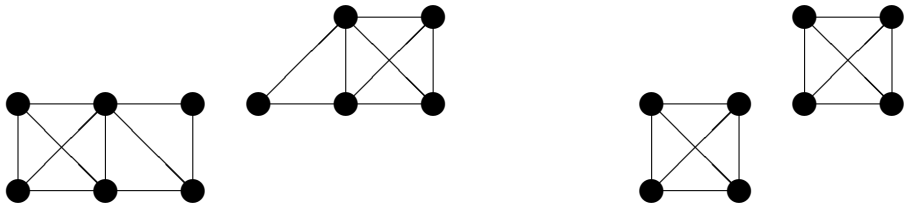
\begin{figure}
\centering
\begin{picture}(10, 3)
\put(0, 0.5){\circle*{0.27}}
\put(0, 1.5){\circle*{0.27}}
\put(1, 0.5){\circle*{0.27}}
\put(1, 1.5){\circle*{0.27}}
\put(2, 0.5){\circle*{0.27}}
\put(2, 1.5){\circle*{0.27}}
\put(0, 0.5){\line(1,0){2}}
\put(0, 0.5){\line(0,1){1}}
\put(0, 0.5){\line(1,1){1}}
\put(0, 1.5){\line(1,0){2}}
\put(0, 1.5){\line(1,-1){1}}
\put(1, 0.5){\line(0,1){1}}
\put(1, 1.5){\line(1,-1){1}}
\put(2, 0.5){\line(0,1){1}}
\put(2.75, 1.5){\circle*{0.27}}
\put(3.75, 1.5){\circle*{0.27}}
\put(3.75, 2.5){\circle*{0.27}}
\put(4.75, 1.5){\circle*{0.27}}
\put(4.75, 2.5){\circle*{0.27}}
\put(2.75, 1.5){\line(1,0){2}}
\put(2.75, 1.5){\line(1,1){1}}
\put(3.75, 1.5){\line(0,1){1}}
\put(3.75, 1.5){\line(1,1){1}}
\put(3.75, 2.5){\line(1,0){1}}
\put(3.75, 2.5){\line(1,-1){1}}
\put(4.75, 1.5){\line(0,1){1}}
\put(7.25, 0.5){\circle*{0.27}}
\put(7.25, 1.5){\circle*{0.27}}
\put(8.25, 0.5){\circle*{0.27}}
\put(8.25, 1.5){\circle*{0.27}}
\put(7.25, 0.5){\line(1,0){1}}
\put(7.25, 0.5){\line(0,1){1}}
\put(7.25, 0.5){\line(1,1){1}}
\put(7.25, 1.5){\line(1,0){1}}
\put(7.25, 1.5){\line(1,-1){1}}
\put(8.25, 0.5){\line(0,1){1}}
\put(9, 1.5){\circle*{0.27}}
\put(9, 2.5){\circle*{0.27}}
\put(10, 1.5){\circle*{0.27}}
\put(10, 2.5){\circle*{0.27}}
\put(9, 1.5){\line(1,0){1}}
\put(9, 1.5){\line(0,1){1}}
\put(9, 1.5){\line(1,1){1}}
\put(9, 2.5){\line(1,0){1}}
\put(9, 2.5){\line(1,-1){1}}
\put(10, 1.5){\line(0,1){1}}
\end{picture}
\caption{Blocks of the $3$-core (left) and $3$-cores of the blocks (right).}
\label{fig:blocks}
\end{figure}

It has been noted that many optimisation problems on graphs can be simplified by working on the blocks
of the graph instead of the original graph; see, e.g., \cite{hochbaum1993should}. The following proposition
shows that this is also the case for the 2L-PP.
\begin{proposition}
For any graph $G$, the cost of the optimal 2L-PP solution is equal to the sum of the costs
of the optimal solutions of the 2L-PP instances given by its blocks.
\end{proposition}
\begin{proofsketch}
If $G$ is disconnected, then the 2L-PP trivially decomposes into one 2L-PP instance for each
connected component.  So assume that $G$ is connected but not biconnected. Let $v \in V$ be an
articulation point, and let $V_1, \ldots, V_t$ be the vertex sets of the connected components of
$G[V \setminus \{v\}]$. Now consider the subgraphs $G[V_i \cup \{v\}]$ for $i = 1, \ldots, t$.
Let $S(i)$ be the optimal solution to the 2L-PP instance on $G[V_i \cup \{v\}]$, represented as a
proper $kk'$-colouring of $V_i \cup \{v\}$, and let $c(i)$ be its cost. Since each edge of $E$ appears in
exactly one of the given subgraphs, the quantity $\sum_{i=1}^t c(i)$ is a lower bound on the cost of the
optimal 2L-PP solution on $G$. Moreover, by symmetry, we can assume that node $v$ receives
colour $1$ in $S(1), \ldots, S(t)$. Now, for a given $u \in V \setminus \{v\}$, let $i(u) \in \{1, \ldots, t\}$
be the unique integer such that $u \in V_{i(u)}$. We can now construct a feasible solution to the 2L-PP
on $G$ by giving node $v$ colour $1$, and giving each other node $u \in V \setminus \{v\}$ the colour that
it has in $S(i(u))$. The resulting 2L-PP solution has cost equal to $\sum_{i=1}^t c(i)$, and is therefore
optimal.
\end{proofsketch}

We call the replacement of a graph with its blocks {\em block decomposition}. Interestingly, a graph
which is not $k$-core reducible may have blocks which are; see again Fig.~\ref{fig:blocks}.  This
leads us to apply $k$-core reduction and block decomposition recursively, until no more reduction
or decomposition is possible. 

At the end of this procedure, we have a collection of induced subgraphs of $G$ which are
biconnected and not $k$-core reducible, and we can solve the 2L-PP on each subgraph
independently. Given the optimal solutions for each subgraph, we can reconstruct an optimal
solution for the original graph by recursively constructing solutions for the predecessor graph
of a reduction or decomposition.

\subsection{Cut-and-branch algorithm}
\label{sub:alg-cut-and-branch}

For each remaining subgraph, we now run our cut-and-branch algorithm. Let $G'=(V',E')$
be the given subgraph. We set up an initial trivial LP, with only one variable $y_e$ for each
edge $e \in E'$, and run a cutting-plane algorithm based on $y$-clique inequalities. Next, we
add the $z$ variables and run another cutting-plane algorithm based on $z$-clique and $yz$-clique
inequalities. Finally, we add the $x$-variables and run branch-and-bound. The full procedure is
detailed in Algorithm~\ref{alg:cut-and-branch}.

\begin{algorithm}
\caption{Cut-and-branch algorithm to solve 2L-PP}
\SetKwInOut{Input}{input}
\SetKwRepeat{Do}{do}{while}
\SetKwInOut{Output}{output}
\Input{subgraph $G'$, 2L-PP problem parameters $k$, $k'$,  continuous
parameter $\epsilon > 0$, integer parameter $t > 0$ }
Enumerate all cliques in $G'$ of size greater than $k$\;
Construct the (trivial) LP relaxation $\min \sum_{e \in E'} y_e$ s.t.\ $y$ non-negative\;
\Do{Violated inequalities found}{
  Solve LP relaxation\;
  Search for violated $y$-clique inequalities \eqref{eq:y-clq}\;
  If any are found, add the $t$ most violated ones to the LP\;
}
Delete all $y$-clique inequalities with slack greater than $\epsilon$\;
Change objective function to (\ref{eq:our-obj}) and add one non-negative variable $z_e$ for all $e \in E'$\;
\Do{Violated inequalities found}{
  Solve LP relaxation\;
  Search for violated $z$-clique inequalities \eqref{eq:z-clq} and $yz$-clique
 inequalities \eqref{eq:yz}\;
  If any are found, add the $t$ most violated ones to the LP\;
}
Delete all inequalities with slack greater than $\epsilon$\;
Add $x$ variables and constraints (\ref{eq:our-x})--(\ref{eq:our-bin-z}), and apply symmetry-breaking\;
Solve resulting 0-1 LP with branch-and-bound\;
\label{alg:cut-and-branch}
\end{algorithm}

The key feature of this approach is that the LP is kept as small as possible throughout the
course of the algorithm. Indeed, (a) the $z$ and $x$ variables are added to the problem only
when they are needed, (b) only a limited number of constraints are added in each cutting-plane
iteration, and (c) slack constraints are deleted after each of the two cutting-plane algorithms
has terminated. The net result is that both cutting-plane algorithms run very efficiently, and so
does the branch-and-bound algorithm at the end.

We now make some remarks about the separation problems for the three kinds of clique
inequalities. Since all three separation problems seem likely to be $\NP$-hard, we initially planned
to use greedy separation heuristics, in which the set $C$ is enlarged one node at a time. We were
surprised to find, however, that it was feasible to solve the separation problems exactly, by brute-force
enumeration, for typical 2L-PP instances encountered in our application. The reason is that the original
graph $G$ tends to be fairly sparse in practice, and each subgraph $G'$ generated by our preprocessor
tends to be fairly small. Accordingly, after the preprocessing stage, we use the Bron-Kerbosch algorithm \cite{Bron1973} to enumerate all maximal cliques in each subgraph $G'$. It is then fairly easy to
solve the separation problems by enumeration, provided that one takes care not to examine the same
clique twice in a given separation call. We omit details, for brevity.

We remark that, although an arbitrary graph with $n$ nodes can have as many as $3^{\frac{n}{3}}$
maximal cliques \cite{moon1965}, a graph in which all nodes have degree at most $d$ can have at most
$(n-d)3^{\frac{d}{3}}$ of them \cite{eppstein2010}.

\section{Computational Experiments} \label{se:experiments}

We now present the results of some computational experiments. In
Subsection~\ref{sub:exp-graphs}, we describe how we constructed the graphs used in our
experiments. In Subsection~\ref{sub:exp-pre}, we test the preprocessing algorithm
for different values of $k$. In Subsection~\ref{sub:exp-cuts}, we present results from the
cutting-plane algorithms. Finally, in Subsection~\ref{sub:exp-overall}, we study the performance
of the algorithm as a whole.

Throughout these experiments, the value of $k$ varies between 2 and 5, while for simplicity,
we fix $k'=2$, since the value of $k'$ does not affect the performance of the graph
preprocessing algorithm. All experiments have been run on a high performance computer
with an Intel 2.6 GHz processor and using 16 cores. Graph preprocessing and clique
enumeration was done using {\tt igraph} \cite{csardinepusz2006} and the linear and integer
programs were solved using {\tt Gurobi} v.6.5 \cite{gurobi}.

\subsection{Graph construction} \label{sub:exp-graphs}
The strength of a signal at a receiver decays in free space at a rate
inversely proportional to the square of the distance from the
transmitter, but in real systems often at a faster rate due to the
presence of objects blocking or scattering the waves.  Therefore,
beyond a certain distance, two transceivers can no longer hear each
other, and therefore there cannot be a direct conflict between their
IDs.  In our application, however, a conflict also occurs if a pair of
devices have a neighbour in common. (Essentially, this is because each
device needs to be able to tell its neighbours apart.)

Accordingly, we initially constructed our graphs as follows. We first sample a specified number
of points uniformly on the unit square. Edges are created between pairs of points if they
are within a specified radius of each other. (This yields a so-called {\em disk graph};
see, e.g., \cite{lu2008survey}.) The graph is then augmented with edges between
pairs of nodes which have a neighbour in common. (In other words, we take the {\em square}\/
of the disk graph.) This construction is illustrated in Figure~\ref{fig:nbr-graph}.

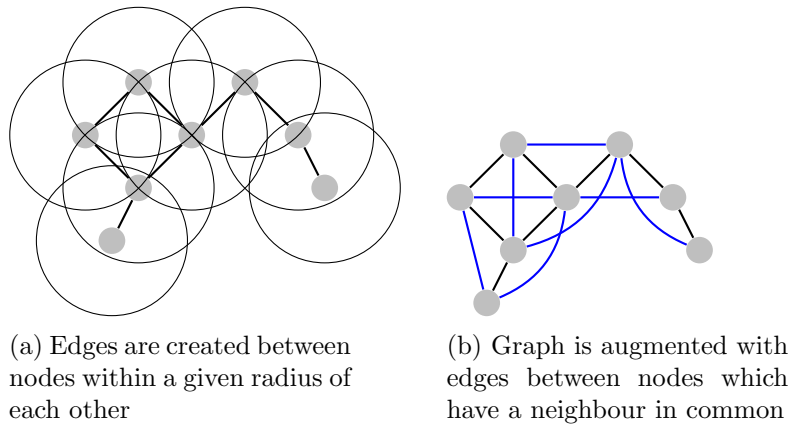
\begin{figure}
  \centering
  \begin{subfigure}[b]{0.3\textwidth}
    \begin{tikzpicture}[scale=0.7, auto,swap]
  \foreach \pos/\name in {{(0.5,0)/a}, {(0,2)/b}, {(1,1)/c},
                          {(1,3)/d}, {(2,2)/e}, {(3,3)/f}, 
                          {(4,2)/g}, {(4.5,1)/h}}
        \node[vertex] (\name) at \pos {};

  \foreach \name in {a,b,c,d,e,f,g,h}
     \draw (\name) circle (1.42);

  \foreach \source/ \dest in {a/c, b/c, b/d, c/e, d/e,
                              e/f, f/g, g/h}
        \path[edge] (\source) -- (\dest);
\end{tikzpicture}

    \caption{Edges are created between nodes within a given radius of each other}
  \end{subfigure}
  \hspace{1cm}
  \begin{subfigure}[b]{0.3\textwidth}
    \begin{tikzpicture}[scale=0.7, auto,swap]
  \foreach \pos/\name in {{(0.5,0)/a}, {(0,2)/b}, {(1,1)/c},
                          {(1,3)/d}, {(2,2)/e}, {(3,3)/f}, 
                          {(4,2)/g}, {(4.5,1)/h}}
        \node[vertex] (\name) at \pos {};

  \foreach \source/ \dest in {a/c, b/c, b/d, c/e, d/e,
                              e/f, f/g, g/h}
        \path[edge] (\source) -- (\dest);

  \path[edge, color=blue] (a) -- (b);
  \path[edge, color=blue] (b) -- (e);
  \path[edge, color=blue] (c) -- (d);
  \path[edge, color=blue] (d) -- (f);
  \path[edge, color=blue] (e) -- (g);
  \draw[thick, color=blue] (a) to [bend right=30] (e);
  \draw[thick, color=blue] (c) to [bend right=30] (f);
  \draw[thick, color=blue] (f) to [bend right=30] (h);
\end{tikzpicture}

    \caption{Graph is augmented with edges between nodes which have a neighbour in common}
  \end{subfigure}
  \caption{Construction of neighbourhood graph}
  \label{fig:nbr-graph}
\end{figure}

It turned out, however, that neighbourhood graphs constructed in this way yielded extremely
easy 2L-PP instances.
The reason is that nodes near to the boundary of the square tend to have small degree, which
causes them to be removed during the preprocessing stage. This in turn causes their neighbours
to have small degree, and so on. In order to create more challenging instances, and to avoid
this ``boundary effect", we decided to use a torus topology to calculate distances between points
in the unit square before constructing the graphs.

\subsection{Preprocessing} \label{sub:exp-pre}

In order to understand the potential benefits of the preprocessing stage, we have
calculated the effect of preprocessing on our random neighbourhood graphs for different values 
of $k$ and disk radius, while fixing $n=100$. In particular, for radius $0.01, \ldots, 0.2$
and $k = 3, 4, 5$, we calculate the mean proportion of edges eliminated and the mean proportion
of vertices in the largest remaining component. The results are shown in Figure~\ref{fig:preprocess}.
The means were estimated by simulating 1000 random graphs for each pair of parameters.

\begin{figure}[h]
  \centering
  \includegraphics[width=0.75\textwidth]{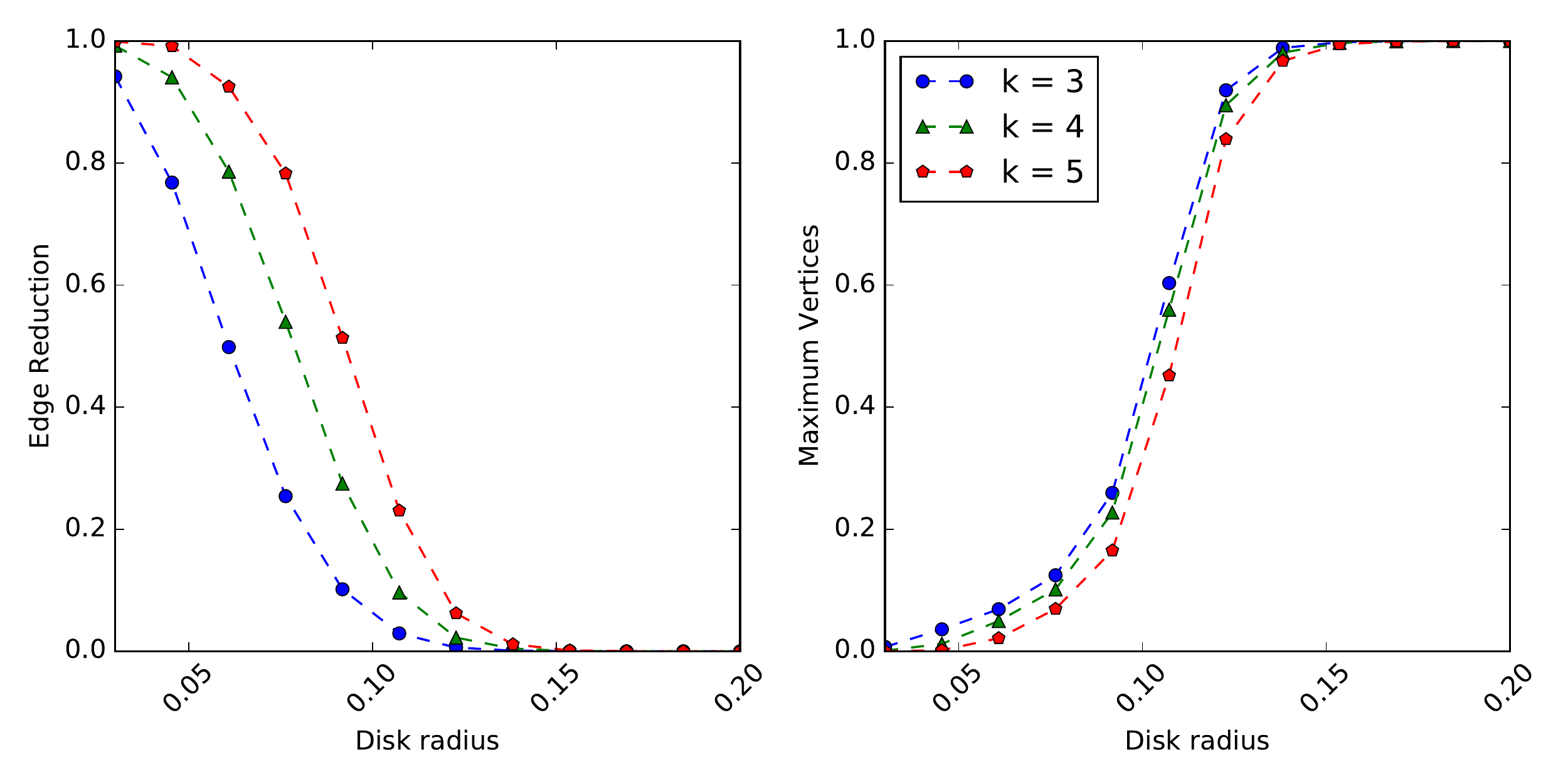}  
  \caption{Preprocessing of random neighbourhood graphs}
  \label{fig:preprocess}
\end{figure}

The results show that, for all values of $k$, the proportion of edges eliminated decays to zero
rather quickly as the disk radius is increased. The proportion of nodes in the largest component
also tends to one as the radius is increased, but at a slower rate than the convergence for edge
reduction.

In order to gain further insight, we also explored how the average degree in our random
neighbourhood graphs depends on the number of nodes and the disk radius. The results are
shown in Figure~\ref{fig:avg-deg}. A comparison of this figure and the preceeding one indicates
that, as one might expect, preprocessing works best when the average degree is not much
larger than $k$.

\begin{figure}[h]
  \centering
  \includegraphics[width=4in]{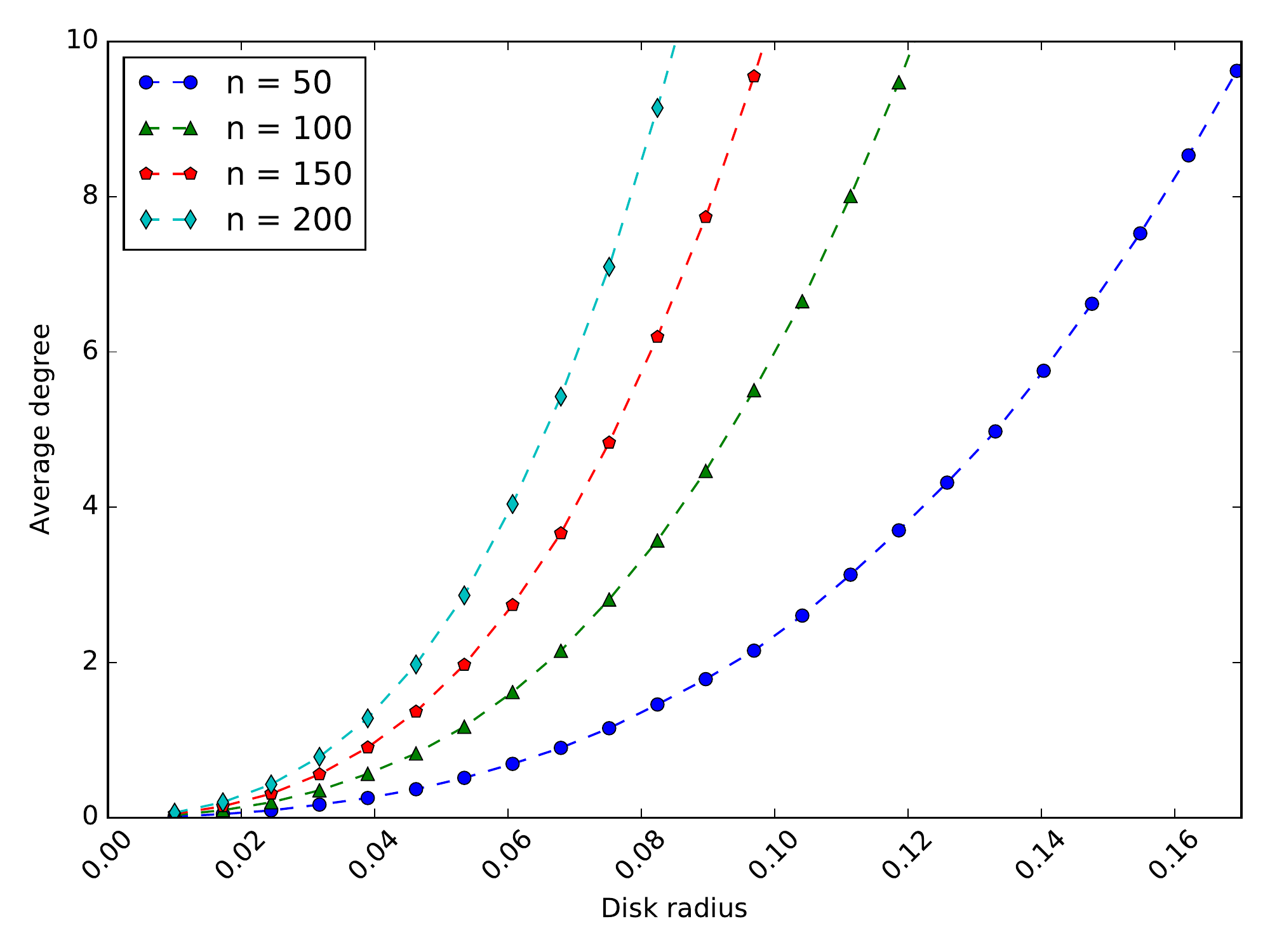}
  \caption{Mean average degree of random neighbourhood graphs}
  \label{fig:avg-deg}
\end{figure}

\subsection{Cutting planes} \label{sub:exp-cuts}

Next, we present some results obtained with the cutting-plane algorithms. For $n=100$, and
for disk radii ranging from $0.08$ to $0.14$, we constructed 50 random neighbourhood graphs.
The same sets of points were used to construct the graphs for each disk radius. We also
considered three different values of $k$, namely $2,3$ and $4$. The weights $w, w'$ were both
set to $1$ for simplicity. For each of the resulting 1050 2L-PP instances,
we ran the preprocessing algorithm, followed by the cut-and-branch algorithm.

As mentioned in Subsection \ref{sub:th-valid}, the $yz$-clique inequalities (\ref{eq:yz}) turned
out to be more effective at improving the lower bound than the $z$-clique inequalities
(\ref{eq:z-clq}). In order to make this clear to the reader, we have divided the second
cutting-plane phase into two sub-phases, whereby $yz$-clique inequalities are added in the
first subphase and $z$-clique inequalities are added in the second subphase.

In Table~\ref{tab:opt-gap}, we present the average gap between the lower bound and optimum,
expressed as a proportion of the optimum, after each of the three kinds of cuts have been
added. We see that, for each value of $k$, the gap increases as we enlarge the disk
radius. Nevertheless, the gap always remains below 5\% for the radii considered. Note that
the $z$-clique inequalities help only when $k=2$. This may be because $z$-clique inequalities
are defined only for cliques containing more than $kk'$ nodes, and not many such cliques
are present when $k > 2$.

In Table~\ref{tab:num-cuts} we present the average number of cuts added during each
cutting-plane phase. We see that the number of cuts increases with the disk radius. This is
to be expected, since, when the radius is large, there are more large cliques present.
On the other hand, the number of cuts decreases as the value of $k$ increases. The
explanation for this is that $y$-clique inequalities are defined only for cliques containing more
than $k$ nodes, and the number of such cliques decreases as $k$ increases.

\begin{landscape}
\newgeometry{left=2cm,bottom=2.0cm, right=2cm, top=5cm}
\begin{table}[h]
  \centering
  \begin{tabular}{l|ccc|ccc|ccc|}
\toprule
& \multicolumn{3}{c}{k = 2} & \multicolumn{3}{c}{k = 3} & \multicolumn{3}{c}{k = 4} \\
\midrule
Radius& y-cut gap& yz-cut gap& z-cut gap& y-cut gap& yz-cut gap& z-cut gap& y-cut gap& yz-cut gap& z-cut gap\\
\midrule
0.08& 0.04& 0.01& 0.01& 0.01& 0.01& 0.01& 0.01& 0.01& 0.01 \\
0.09& 0.05& 0.01& 0.01& 0.01& 0.01& 0.01& 0.01& 0.01& 0.01 \\
0.10& 0.07& 0.02& 0.01& 0.02& 0.01& 0.01& 0.00& 0.00& 0.00 \\
0.11& 0.09& 0.02& 0.01& 0.04& 0.03& 0.03& 0.05& 0.05& 0.05 \\
0.12& 0.11& 0.02& 0.01& 0.05& 0.03& 0.03& 0.02& 0.02& 0.02 \\
0.13& 0.14& 0.03& 0.03& 0.07& 0.04& 0.04& 0.03& 0.03& 0.03 \\
0.14& 0.16& 0.04& 0.03& 0.08& 0.04& 0.04& 0.04& 0.04& 0.04 \\
\bottomrule
\end{tabular}

  \caption{Average optimality gap at the end of each cutting-plane phase}
  \label{tab:opt-gap}
\end{table}
\vspace{-0.6cm}
\begin{table}[h]
  \centering
  \begin{tabular}{l|ccc|ccc|ccc|}
\toprule
& \multicolumn{3}{c}{k = 2} & \multicolumn{3}{c}{k = 3} & \multicolumn{3}{c}{k = 4} \\
\midrule
Radius& \# y-cuts& \# yz-cuts& \# z-cuts& \# y-cuts& \# yz-cuts& \# z-cuts& \# y-cuts& \# yz-cuts& \# z-cuts\\
\midrule
0.08& 45.6& 1.2& 0.4& 9.3& 0.0& 0.0& 1.2& 0.0& 0.0 \\
0.09& 72.5& 3.1& 1.1& 18.4& 0.0& 0.0& 3.0& 0.0& 0.0 \\
0.10& 110.0& 8.7& 3.0& 33.0& 0.1& 0.0& 8.0& 0.0& 0.0 \\
0.11& 155.0& 17.7& 7.4& 58.1& 0.7& 0.2& 16.1& 0.0& 0.0 \\
0.12& 198.8& 30.9& 16.1& 85.2& 2.4& 0.8& 27.7& 0.0& 0.0 \\
0.13& 269.9& 52.9& 24.7& 116.0& 3.9& 1.2& 39.5& 0.0& 0.0 \\
0.14& 376.4& 87.7& 36.1& 173.5& 8.6& 3.0& 62.9& 0.1& 0.0 \\
\bottomrule
\end{tabular}

  \caption{Average number of cuts added to linear relaxation during each cutting-plane phase}
  \label{tab:num-cuts}
\end{table}
\vspace{-0.6cm}
\begin{table}[h]
  \centering
  \begin{tabular}{l|cc|cc|cc|}
\toprule
& \multicolumn{2}{c}{k = 2} & \multicolumn{2}{c}{k = 3} & \multicolumn{2}{c}{k = 4} \\
\midrule
Radius& CP time (s)& BB time (s)& CP time (s)& BB time (s)& CP time (s)& BB time (s)\\
\midrule
0.08& 0.01& 0.03& 0.00& 0.01& 0.00& 0.00 \\
0.09& 0.02& 0.06& 0.00& 0.04& 0.00& 0.01 \\
0.10& 0.02& 0.13& 0.01& 0.06& 0.00& 0.01 \\
0.11& 0.02& 0.26& 0.01& 0.21& 0.00& 0.05 \\
0.12& 0.03& 1.40& 0.01& 0.34& 0.00& 0.09 \\
0.13& 0.05& 6.70& 0.02& 4540.86& 0.00& 0.87 \\
0.14& 0.09& 156.68& 0.03& 746.36& 0.01& 62.28 \\
\bottomrule
\end{tabular}

  \caption{Average cutting plane (CP) and branch-and-bound (BB) times}
  \label{tab:time}
\end{table}
\end{landscape}

\subsection{Overall algorithm} \label{sub:exp-overall}

Finally, we report results obtained with the overall exact algorithm. Table~\ref{tab:time}
presents the running times of the cutting plane and branch-and-bound phases of the algorithm.
As expected, the running time increases as the disk radius increases and the graph becomes more dense. On the other hand,
perhaps surprisingly, it decreases as $k$ increases. This is partly because the preprocessing
stage removes more nodes and edges when $k$ is larger but also because fewer edge conflicts occur when we use more colours. We also see that, for all values
of $k$ and disk radii, the running time of the cutting-plane stage is negligible compared
with the running time of the branch-and-bound stage. This is so, despite the fact that we are
using enumeration to solve the separation problems exactly.

\section{Conclusions} \label{se:conclusion}

In this paper we have defined and tackled the 2-level graph partitioning problem which,
as far as we are aware, has not previously been addressed in the optimization or data-mining
literature. Although this model was motivated by a problem in telecommunications it may
have other applications, such as in hierarchical clustering.

The instances encountered in our application were characterised by
small values of $k$ and $k'$, and large, sparse graphs. For instances
of this kind, we proposed a solution approach based on aggressive
preprocessing of the original graph, followed by a novel multi-layered cut-and-branch
scheme, which is designed to keep the LP as small as possible at each stage.
Along the way, we also derived new valid inequalities and symmetry-breaking constraints.

One possible topic for future research is the derivation of additional families of
valid inequalities, along with accompanying separation algorithms (either exact or heuristic).
Another interesting topic is the ``dynamic" version of our problem, in which devices
are switched on or off from time to time.

\section*{Appendix}

\noindent
{\bf Proof of Theorem \ref{th:yz-valid}:}
For a clique $C\subset V$ let $\mathcal{C}: C \to \{0, \ldots, kk'-1\}$ be a $kk'$-coloring of $C$. For $c = 0, \ldots, k-1$, let $S_{c}$ denote
$\big\{ v \in C : \, \mathcal{C}(v) \in \{c, c+ k, \ldots, c+(k'-1)k\} \big\}$ and for $c=0,\ldots,kk'-1$ let $W_{c}$ denote $\{ v \in C : \, \mathcal{C}(v) = c \}$. Note that, for $c=0,\ldots,k-1$,
\begin{equation}
  \label{eq:union-wc}
  \bigcup_{t=0}^{k'-1} W_{c+tk} = S_{c}.
\end{equation}
Fix $c\in \{0,\ldots, k-1\}$. By definition we have
\begin{equation*}
  y(S_{c}) = \binom{|S_{c}|}{2},
\end{equation*}
and
\begin{equation*}
  z(S_{c}) = \sum_{t=0}^{k'-1} \binom{|W_{c+tk}|}{2}.
\end{equation*}
Suppose $|S_{c}| = p_{c}k' + r_{c}$ where $0\leqslant  r_{c} < k'$. Then the second summation is minimized
when
\begin{align*}
  |W_{c+tk}| & = p_{c} + 1 & \text{ for } t=0,\ldots,r-1\\
                    & = p_{c}  & \text { for } t=r,\ldots,k'-1.
\end{align*}
Hence,
\begin{align*}
  z(S_{c}) &\geqslant r_{c} \binom{p_{c} + 1}{2} + (k'-r_{c})\binom{p_{c}}{2}\\
  &=\frac12\left(k'p_{c}^{2} + 2p_{c}r_{c} - k'p_{c}\right)\\
  &=\frac{1}{2k'}\left(k'^{2}p_{c}^{2} + 2 k' p_{c}r_{c} - k'^{2}p_{c}\right)\\
  &=\frac{1}{2k'}\left((k' p_{c} + r)(k' p_{c} + r - 1) + k'p_{c} + r_{c} - r_{c}^{2} + k'^{2}p_{c}\right)\\
  &=\frac{1}{k'}\binom{|S_{c}|}{2} + \frac{1}{2k'}\left(|S_{c}| - (k'^{2}p_{c} + r_{c}^{2})\right)\\
  &=\frac{1}{k'}y(S_{c}) + \frac{1}{2k'}\left(|S_{c}| - \left(k'(|S_{c}| - r_{c}) + r_{c}^{2}\right)\right)\\
  &=\frac{1}{k'}y(S_{c}) - \frac{1}{2k'}\left((k'-1)|S_{c}| - r_{c}(k' - r_{c})\right).\\
\end{align*}
Now, 
\begin{align*}
  z(C) &= \sum_{c=0}^{k-1} z(S_{c})\\
  &\geqslant \sum_{c=0}^{k-1}\left(\frac{1}{k'}y(S_{c}) - \frac{1}{2k'}\left((k' - 1)|S_{c}| - r_{c}(k' - r_{c})\right)\right)\\
  &\geqslant \frac{1}{k'}y(C) - \frac{1}{2k'}\left((k'-1)|C| - \sum_{c=0}^{k-1}r_{c}(k' - r_{c})\right).
\end{align*}
The last expression is minimized when $\sum_{c=0}^{k-1}r_{c}(k' - r_{c})$ is minimized. Noting that
$\sum_{c=0}^{k-1} r_{c} \bmod k' \geqslant |C| \bmod k' = R$, we see that this expression is minimized when
$r_0 = R$ and $r_{c} = 0$ for $c=1,\ldots,k-1$. Hence,
\begin{equation*}
  z(C) \geqslant \frac{1}{k'}y(C) - \frac{1}{2k'}\left((k'-1) |C| - R(k'-R)\right),
\end{equation*}
which is equivalent to the $(y,z)$-clique inequality (\ref{eq:yz}).\hfill$\Box$\\

\noindent
{\bf Proof of Theorem \ref{th:z-redundant}}:
Let $R = |C| \bmod kk'$. For the case $R=0$, the $z$-clique inequality is implied by other $z$-clique
inequalities (see \cite{chopra1993partition}). For the cases $R=1$ and $R = kk'-1$, we show that the
associated $z$-clique inequality is implied by $y$-clique and $(y,z)$-clique inequalities.

First, suppose that $R=1$, i.e., that $|C| = Tkk' + 1$ for some positive integer $T$.
In this case, the $z$-clique inequality takes the form:
\begin{equation} \label{eq:z-clique-case-1}
z(C) \, \geqslant \, \binom{T+1}{2} + \binom{T}{2} (kk' - 1) \: = \: k'(k/2)(T^2 -T) + T.
\end{equation}
The $y$-clique inequality on $C$ takes the form:
\begin{equation} \label{eq:y-clique-case-1}
y(C) \,\geqslant \, \binom{k'T+1}{2} + \binom{k'T}{2} (k-1) \: = \: k'(k/2)(k'T^2 - T) + k'T,
\end{equation}
and the $(y,z)$-clique inequality on $C$ takes the form:
\begin{equation} \label{eq:yz-clique-case-1}
k' z(C) - y(C) \geqslant - Tk \binom{k'}{2} = k'(k/2) (T - k'T).
\end{equation}
Adding inequalities \eqref{eq:y-clique-case-1} and \eqref{eq:yz-clique-case-1}, and dividing
the resulting inequality by $k'$ yields the $z$-clique inequality \eqref{eq:z-clique-case-1}.

Second, suppose that $R = kk' - 1$, i.e., that $|C| = k k' T + (kk'-1)$ for some positive integer $T$.
In this case, the $z$-clique inequality takes the form:
\begin{equation} \label{eq:z-clique-case-2}
z(C) \, \geqslant \, \binom{T+1}{2} (kk' - 1) + \binom{T}{2} \: = \: k'(k/2)(T+1)T - T.
\end{equation}
The $y$-clique inequality on $C$ takes the form:
\begin{eqnarray}
\nonumber
y(C) \geqslant & & \binom{Tk'+k'}{2}(k-1) + \binom{Tk'+k'-1}{2} \\
\label{eq:y-clique-case-2}
         = & & k'(k/2)(T+1)(Tk'+k'-1) - (Tk' + k' - 1).
\end{eqnarray}
and the $(y,z)$-clique inequality on $C$ can be written as follows:
\begin{eqnarray}
\nonumber
k' z(C) - y(C) \geqslant & & -(Tk+k-1)\binom{k'}{2} - \binom{k'-1}{2} \label{eq:yz-clique-case-2} \\
    = &  &  k'(k/2)(T+1)(1-k') + (k'-1).
\end{eqnarray}
Adding inequalities \eqref{eq:y-clique-case-2} and \eqref{eq:yz-clique-case-2}, and dividing the
resulting inequality by $k'$ yields the $z$-clique inequality \eqref{eq:z-clique-case-2}.\hfill$\Box$\\

\noindent
{\bf Proof of Theorem \ref{th:symmetry}:}
It is sufficient to show that for any colouring $\mathcal{C}:V \to \{0,\ldots,kk' - 1\}$
there exists a objective-preserving permutation $\sigma : \{0,\ldots,kk'-1\} \to \{0,\ldots,kk'-1\}$ such that
 \begin{equation}
   \label{eq:perm-sat}
   \phi(\sigma\circ\mathcal{C}(v)) > v \qquad \text{for } v=1,\ldots,n.
 \end{equation}
By objective-preserving we mean that:
 \begin{equation*}
   c_{1} \bmod k = c_{2} \bmod k \Longleftrightarrow \sigma(c_{1}) \bmod k = \sigma(c_{2}) \bmod k \qquad \text{for all } 0 < c_{1},c_{2}<kk'.
 \end{equation*}
We prove that such a permutation exists by induction on the number of nodes in $V$.
For the case $n=1$, the result holds trivially.

Suppose that the result holds $n \leq l$, and let us now consider the case $n=l+1$.
Using our induction hypothesis, we suppose, for notational convenience,
that $\phi(\mathcal{C}(v)) < v$ for $v = 1,\ldots l$. Now, let $\mathcal{C}(l+1) = pk + r$.
If $\phi(\mathcal{C}(l+1)) = p + r < l + 1$ then the colouring already satisfies
the required condition so we assume that this is not true.

We consider two cases. In the first case, suppose that $\mathcal{C}(v) \bmod k \neq r$
for all $v = 1,\ldots,l$. This can be split into two further subcases. In the first subcase,
we assume that $l+1 \leqslant k$. We now define an objective-preserving permutation which 
assigns the colour $l$ to node $l+1$. Note that by induction hypothesis, none of the nodes
$v = 1,\ldots, l$ are assigned to colour $l$. Define permutations $\sigma_{1}$ and
$\sigma_{2}$ as follows:
 \begin{align*}
   \sigma_{1}:& 
   \begin{cases}
     tk + r \mapsto tk + l & \\
     tk + l \mapsto tk + r & \text{ for } t = 0, \ldots, k'-1\\
     c \mapsto c & \text{ otherwise,}
   \end{cases}\\
   \sigma_{2}:&
   \begin{cases}
     pk + l \mapsto l & \\
     l \mapsto pk + l & \\
     c \mapsto c & \text{ otherwise,}
   \end{cases}
 \end{align*}
then $\sigma_{2}\circ\sigma_{1}\circ\mathcal{C}(l+1) = l$ and so
we have $\phi\left(\sigma_{2}\circ\sigma_{1}\circ\mathcal{C}(l+1)\right) = l < l+1$
as required.
  
In the second subcase we assume that $l+1 > k$. In this case, the following
objective-preserving permutation can be used:
 \begin{equation*}
   \sigma:
   \begin{cases}
     pk + r \mapsto r & \\
     r \mapsto pk + r & \\
     c \mapsto c & \text{ otherwise.}
   \end{cases}
 \end{equation*}
Then, $\phi(\sigma\circ\mathcal{C}(l+1)) = r < k < l+1$ as required.

In the second case, we suppose that the set $C_{rl} := \{1\leq v \leq l: \mathcal{C}(v) \bmod k = r\}$
is non-empty, and let $u = \max C_{rl}$. Then, for some $0 \leq s < p$ we have $\mathcal{C}(u) = sk + r$.
We now define the required permutation:
 \begin{equation*}
   \sigma:
   \begin{cases}
     pk + r \mapsto (s+1)k + r & \\
     (s+1)k + r \mapsto pk + r & \\
     c \mapsto c & \text{otherwise.}
   \end{cases}
 \end{equation*}
Then, $\phi\left(\sigma\circ\mathcal{C}(l+1)\right) = (s+1) + r < u + 1 \leq l+1$
as required, where the first inequality follows from our induction hypothesis.\hfill$\Box$

\bibliographystyle{alpha}
\bibliography{bibliography}

\end{document}